\documentclass[a4paper,11pt,reqno]{amsart}
\usepackage[sf,wide,pthms,tikz]{gewoehn} 
\usepackage{cases}
\usepackage[bookmarks=false, pdfpagemode=UseNone, pagebackref, linktocpage=true,
colorlinks=true, linkcolor=RoyalBlue, citecolor=BrickRed,
urlcolor=RoyalBlue]{hyperref}

\begin{document}

\title[Rigidity of hypersurfaces of spherical space forms]{Homotopical and
topological rigidity of \\ hypersurfaces of spherical space forms}
\author{Pedro Z\"{u}hlke}
\maketitle

\begin{abstract}
The first main result is a topological rigidity theorem for complete immersed
hypersurfaces of spherical space forms from which similar theorems due
to Wang/Xia and Longa/Ripoll can be derived.  Under certain sharp conditions
on the principal curvatures of such a hypersurface $ f \colon N^n \to M^{n+1} $
($ n\ge 2 $), it asserts that the universal cover of $ N $ must be diffeomorphic
to the $ n $-sphere $ \Ss^n $, and provides an upper bound for the order of the
fundamental group of $ N $ in terms of that of $ M $.  In particular, if $ M =
\Ss^{n+1} $, then $ N $ is diffeomorphic to $ \Ss^n $ and either $ f $ or its
Gauss map is an embedding. 	

Let $ J \subs (0,\pi) $ be any interval of length less than $ \frac{\pi}{2} $.
The second main result constructs a weak homotopy equivalence between the space
of all complete immersed hypersurfaces of $ M $ with principal curvatures in $
\cot (J) $ and the twisted product of $ \big( \Ga\backslash \SO_{n+2} \big) $
and $ \Diff_+(\Ss^n) $ by $ \SO_{n+1} $, where $ \Ga $ is the fundamental group
of $ M $ regarded as a subgroup of $ \SO_{n+2} $.
	 
Relying on another rigidity criterion due to Wang/Xia, the third main result
constructs a homotopy equivalence between the space of all
complete immersed hypersurfaces of $ \Ss^{n+1} $ whose Gauss maps have image
contained in a strictly convex ball and the same twisted product, with $ \Ga $
the trivial group.  
\end{abstract}



\setcounter{section}{-1}
\section{Introduction}\label{S:introduction}


\begin{ucvn}\label{N:implicit}
	Throughout the article, $ n\geq 2 $ is an integer and \tit{manifolds are
	implicitly assumed to be connected, oriented and smooth}, i.e., of class $
	C^{\infty} $. Maps between manifolds are also assumed to be smooth, and sets
	of such maps are furnished with the $ C^{\infty} $-topology.
\end{ucvn}

\subsection*{Topological rigidity} A classical theorem of J.~Hadamard
\cite{Hadamard} states that a closed surface in the euclidean space $
\E^{3} $ whose Gaussian curvature does not vanish must be embedded as the
boundary of a convex body, and in particular diffeomorphic to the 2-sphere. The
following is the analogue of Hadamard's theorem for hypersurfaces of the sphere
$ \Ss^{n+1} $ (with the standard round metric).

\begin{thm}[do Carmo/Warner, {\cite[thm.~1.1]{CarWar}}]\label{T:doCarmo/Warner}
	Let $ N^n $ be closed and $ f\colon N^n \to
	\Ss^{n+1} $ be an immersion. Suppose that all sectional curvatures of $ f $ 
	are $ \geq 1 $. Then $ N $ is diffeomorphic to $ \Ss^n
	$, $ f $ is an embedding and $ f(N) $ is either totally geodesic (i.e., a
	great hypersphere) or contained in an open hemisphere. In the latter case, $
	f(N) $ is the boundary of a convex body.\footnote{The original statement
	also includes the assertion that $ f $ is geometrically rigid in the sense
	that if $ \bar{f} \colon N^n \to \Ss^{n+1} $ is another immersion inducing
	the same metric on $ N $ as $ f $, then there exists $ Q \in \Oo_{n+2} $
	such that $ \bar{f} = Q \circ f $. For a closely related theorem of wider
	scope, see \cite{Eschenburg}. } 
\end{thm}

Because it will be necessary to consider several immersions $ f\colon
N^n \to M^{n+1} $ from a given manifold $ N $ to a given Riemannian manifold $ M
$ at once, our viewpoint will be that $ N $ is in each case furnished with the
corresponding induced metric. As above, this is reflected in the terminology in
that we speak of, e.g., the principal curvatures of $ f $ (not of $ N $). By a
\tdef{hypersurface} of $ M $ is meant such an immersion, not necessarily an
embedding.

\begin{dfn}[Gauss map, dual]\label{D:Gauss}
	The \tdef{Gauss map} $ \nu = \nu_f \colon N^n \to TM^{n+1} $ of an immersion
	$ f \colon N^n \to M^{n+1} $ is uniquely determined by the condition that
	for all $ p \in N $, $ (u_1,\dots,u_n) $ is a positively oriented
	orthonormal frame in $ TN_p $ if and only if 
	\begin{equation*}
		\big(\?df_p(u_1)\?,\?\dots\?,\?df_p(u_n)\?,\?\nu(p)\? \big)
	\end{equation*}
	is a positively oriented orthonormal frame in $ TM_{f(p)} $.  
	By convention, $ (u_1,\dots,u_{n+1}) $  is positively oriented in $
	T\Ss^{n+1}_p $ if and only if $ (u_1,\dots,u_{n+1},p) $ is positively
	oriented in $ \R^{n+2} $.  
	For $ M $ a spherical space form, the	\tdef{dual} of $ f $ (which need not
	be an immersion) is the map 
	\begin{equation*}
		f^{\star} \colon N^n \to M^{n+1},\quad  p \mapsto \exp_{f(p)}\big(
		\tfrac{\pi}{2}\nu(p) \big).
	\end{equation*} 
	Thus when $ M = \Ss^{n+1} $, $ f^{\star} $ is simply the Gauss map $ \nu $
	of $ f $ regarded as a map into $ \Ss^{n+1} $, and both notations will be
	used. 
\end{dfn}

\begin{dfn}[principal radius, $ J(f) $]\label{D:principal}
	Recall that a hypersphere of metric radius $ r $ in $ \Ss^{n+1} $ has
	principal curvatures equal to $ \pm \cot r $.  A \tdef{principal radius} of
	a hypersurface $ f $ of a spherical space form is an element $ \rho $ of the
	circle $ \R \pmod \pi $ such that $ \cot \rho $ is a principal curvature of
	$ f $.    An \tdef{interval} of $ \R \pmod
	\pi $ is a connected subset thereof.  We denote by $ J(f) \subs \R \pmod \pi
	$ a smallest interval which contains all principal radii of $ f
	$.\footnote{If the set of all principal radii is contained in a
	half-open interval of length $ \frac{\pi}{2} $, then $ J(f) $ is
	uniquely determined; otherwise, it may not be. However, we are only
	interested in the former case.} 
\end{dfn}

For reasons which will be clarified later, the main results are formulated in
terms of principal radii. The following can be regarded as an extension of 
part of \tref{T:doCarmo/Warner}, as well as of similar theorems due to Wang/Xia
and Longa/Ripoll (see \tref{T:Wang/Xia} and \tref{T:Longa/Ripoll}).

\begin{thm}[topological rigidity in $ \Ss^{n+1} $]\label{T:rigidity}
	Let $ f \colon N^n \to \Ss^{n+1} $ be an immersion. Suppose that $ N $ is
	complete \tup(with respect to the metric induced by $ f $\tup) and $
	\length(J(f)) < \frac{\pi}{2} $. Then:
	\begin{enumerate}
		\item [(a)] $ N $ is diffeomorphic to $ \Ss^n $.
		\item [(b)] If $ J(f) $ does not contain $ 0 \pmod \pi $, then $ f $ is
			an embedding.
	\item [(c)] If $ J(f) $ does not contain $ \frac{\pi}{2} \pmod \pi $, then
		the dual of $ f $ is an embedding.
	\end{enumerate}
	In particular, (a)--(c) hold if $ N $ is closed and $ J(f) $ is disjoint
	from $ J(f) + \frac{\pi}{2} $.
\end{thm}

\begin{uexm}\label{E:comparison}
	To compare the hypotheses of \tref{T:doCarmo/Warner} and \tref{T:rigidity},
	let $ N^n $ be closed and $ f \colon N^n \to \Ss^{n+1} $ be an immersion.
	Then
	\begin{equation}\label{E:or} 
		J(f) \subs (0,\tfrac{\pi}{2}] \pmod \pi \quad \text{or} \quad  J(f)
		\subs [\tfrac{\pi}{2},\pi) \pmod \pi 
	\end{equation} 
	if and only if the principal curvatures are all nonnegative or
	nonpositive, respectively.  If either holds, then Gauss' equation
	implies that the sectional curvatures are $ \geq 1 $.
	
	Conversely, if the sectional curvatures are all $ \geq 1 $,
	then \eqref{E:or} must hold by \tref{T:doCarmo/Warner} together with
	Gauss' equation. This is clear if $ f $ is totally geodesic. Otherwise, 
	by convexity, for each $ p \in N$, $
	f(N) $ lies wholly on the side of the tangent hypersphere at $ f(p) $ to
	which $ \nu(p) $ or $ -\nu(p) $ points, and the sign is independent of $
	p $ by connectedness.  Hence no two principal curvatures have
	opposite signs.  
\end{uexm}


\begin{rmks}\label{R:principalcurvatures}In terms of principal curvatures, the
	conditions in \tref{T:rigidity} mean the following: 
	\begin{enumerate}
		\item [(i)]  $ J(f) $ has length less than $
			\frac{\pi}{2} $ and does not contain $ 0 \pmod \pi $
			if and only if all principal curvatures of $ f $ lie in an interval
			$ (a,b) $ with $ ab > -1 $ and $ b \in [0,+\infty] $ (where $ (\pm
			\infty)\?0 =0 $).  Thus part (b) applies only
			to hypersurfaces whose sectional curvatures are greater than some $
			c > 0 $, but for $ c $ as close to 0 as desired.
		\item [(ii)] $ J(f) $ has length less than $ \frac{\pi}{2} $
			and does not contain $ \frac{\pi}{2} \pmod \pi $ if and only if
			all principal curvatures of $ f $ lie in $ (-\infty,a) \cup
			(b,+\infty)  $ with $ ab < -1 $ and $ b \in (0,+\infty) $.  Thus
			part (c) applies only to hypersurfaces of two types: those whose
			sectional curvatures are everywhere strictly greater than 1, and
			those with the property that at every point, some sectional
			curvature is negative. The type depends on whether all principal
			curvatures have the same sign or not.
	\end{enumerate}
\end{rmks}

\begin{exm}\label{E:plane}
	A celebrated theorem of J.\,F.~Adams \cite[thm.~1.1]{Adams} states that it is
	possible to define a smooth $ k $-dimensional frame field over $ \Ss^n $ if and
	only if $ 0 \leq k \leq \rho(n+1) - 1 $, where $ \rho $ is the 
	Radon-Hurwitz function. By orthogonal duality, $ \Ss^n $ supports a smooth $
	k $-dimensional distribution if and only if it supports an $ (n-k)
	$-dimensional distribution. Assume without loss of generality that $ 2k \leq
	n $. In this case, the existence of a $ k $-dimensional distribution is
	equivalent to the existence of a $ k $-dimensional frame field over $ \Ss^n
	$ \cite[thm.~27.16]{Steenrod}. In summary, $ \Ss^n $ supports a $ p
	$-dimensional distribution $ (0 \leq p \leq n) $ if and only if either $ p <
	\rho(n+1) $ or $ n-p < \rho(n+1) $. 	

	Now suppose that the principal curvatures of a closed hypersurface $ f
	\colon N^n \to \Ss^{n+1} $ lie in $ (-\infty,-c^{-1}) \cup (c,+\infty) $ for
	some $ c \in (0,+\infty) $, e.g., they are greater than 1 in absolute value. Let
	$ 0 \leq p \leq n$ be the number of positive principal curvatures at some
	(hence every) point of $ N $. The corresponding principal directions define
	a $ p $-dimensional distribution over $ N $.	But by
	\tref{T:rigidity}, $ N $ is diffeomorphic to $ \Ss^n $.
	In particular, if $ n $ is even, we conclude that either $ p=0 $ or $ p=n $,
	so that both $ f(N) $ and $ f^\star(N) $ must be embedded as boundaries of
	convex bodies. 
	
	For contrast, 
	\begin{equation*}
		f \colon \tfrac{1}{\sqrt{2}}\big(\Ss^{k} \times \Ss^{n-k}\big) \inc
		\Ss^{n+1} \qquad (k,\,n-k \geq 1)
	\end{equation*}
	has $ \pm 1 $ for its principal curvatures, so that $ J(f) =
	\big[\frac{\pi}{4},\frac{3\pi}{4}\big] $ or $
	\big[-\frac{\pi}{4},\frac{\pi}{4}\big] \pmod \pi $. Consideration of normal
	translates of the latter shows that given any interval $
	J $ of length greater than $ \frac{\pi}{2} $, there exists a closed
	hypersurface with principal radii in $ J $ which is not diffeomorphic to a
	sphere. Moreover, setting $ k = 1 $ and pre-composing with a self-map of $ \Ss^1 $
	of degree greater than 1, one can arrange that neither $ f $ nor $
	f^\star $ be injective. In this sense, \tref{T:rigidity} is sharp.
\end{exm}



\begin{thm}[topological rigidity in space forms]\label{T:spherical}
	Let $ M^{n+1} $ be a spherical space form and $ f \colon N^n \to M^{n+1} $
	an immersion. Suppose that $ \length(J(f)) < \frac{\pi}{2} $ and $ N $ is
	complete \tup(with the metric induced by $ f $\tup).  Then the universal cover
	of $ N $ is diffeomorphic to $ \Ss^n $.  Moreover:
	\begin{enumerate}
		\item [(a)]  If $ J(f) $ does not contain $ 0 \pmod \pi $ and
			$ m $ denotes the maximum number of preimages under $ f $ of a point
			in $ f(N) $, then
			$ m\abs{\pi_1(N)} \leq \abs{\pi_1(M)} $.
		\item [(b)] If $ J(f) $ does not contain $ \frac{\pi}{2} \pmod \pi $
			and $ m^\star $ denotes the maximum number of preimages under $
			f^{\star} $ of a point in $ f^{\star}(N) $, then $
			m^\star\abs{\pi_1(N)} \leq \abs{\pi_1(M)} $.
		\item [(c)]
			Let $ \pr_M \colon \Ss^{n+1} \to M $ be the covering projection.  If
			$ f $ is an embedding and $ k $ denotes the number of components of
			$ \pr_M^{-1}(f(N)) \subs \Ss^{n+1} $, then $ k\abs{\pi_1(N)} =
			\abs{\pi_1(M)} $.
	\end{enumerate}
\end{thm}
\begin{urmk}
	Even if $ M $ or $ N $ is not orientable, one can still define unsigned
	principal curvatures $ \ka \in [0,+\infty) $ of an immersion $ f \colon N^n
	\to M^{n+1} $ using local Gauss maps.  Suppose that all of them
	satisfy $ \ka < 1 - \eps $, or all satisfy $ \ka > 1 + \eps $ (for some $ \eps
	> 0 $). It follows from our proof that \tref{T:spherical} still applies in
	these two situations, because any lift $ \te{f}  \colon \te{N} \to
	\Ss^{n+1} $ of $ f $ to the universal cover $ \te{N} $ of $ N $ has the
	property that $ \length(J(\te{f})) < \frac{\pi}{2} $.
	
	With this in mind, \tref{T:spherical}\,(c) implies that one can obtain the
	conclusion of \cite[thm.~2]{LonRip1} while omitting their hypothesis on the
	distance to the cut locus and assuming only completeness, instead of
	closedness.  Indeed, their restriction on the principal curvatures
	immediately implies that they must be greater than $ 1+\eps $ in absolute
	value.\footnote{It should be noted that unlike here, in \cite{LonRip1} the
	term ``hypersurface'' stands for (the image of) an {embedding}, while an
	immersion is called an ``immersed hypersurface''. }
\end{urmk}

\begin{uexms}
	Let $ N^n $ be closed and possibly nonorientable, and $ f \colon N^n \to
	\RP^{n+1} $ be an immersion. If the unsigned principal curvatures are all
	less (resp.~greater) than 1, then $ N $ must be diffeomorphic to $ \Ss^n $
	or $ \RP^n $.  Moreover, if $ f $ (resp.~$ f^\star $) is not an embedding,
	then $ N = \Ss^n $. Compare \cite[cor.~1.1]{LonRip1}.

	When $ n $ is even, the only nontrivial group which acts
	freely on $ \Ss^n $ is $ \ifrac{\Z}{2\Z} $; thus, in the situation of
	\tref{T:spherical}, $ N $ is diffeomorphic to $ \Ss^n $ (or $ \RP^n $
	if nonorientability is allowed, as above).
\end{uexms}
 
\subsection*{Homotopical rigidity}
Let $ f \colon \Ss^n \to \Ss^{n+1} $ be an immersion and suppose that there
exists $ c \in \Ss^{n+1} $ such that $ f(\Ss^n) $ is contained in the open
hemisphere determined by $ c $. A canonical choice would be to take $ c $ as
the (circum)center of $ f(\Ss^n) $.\footnote{See
	\cite[ch.~II.2]{BriHae} for the definition and basic properties of the
	center of bounded subsets of CAT($ \ka $) spaces.} 
Choose any $ Q_f \in \SO_{n+2} $ satisfying $
Q_f(-e_{n+2}) = c $, and let $ \pi $  denote
central projection of the southern hemisphere onto the affine hyperplane $
\R^{n+1} \times \se{-1} $, which we identify with $ \E^{n+1} $.  Define $
\bar f = \pi \circ Q_f^{-1} \circ f $.  With the conventions described in
\dref{D:Gauss}, it is readily proved that the principal curvatures of $ f $ are
positive if and only if those of $ \bar f $ are negative.  Now it is a
consequence of \cite[prop.~4.3]{Zuehlke2} that the space of all immersions $
\Ss^n \to \E^{n+1} $ having negative principal curvatures is homotopy equivalent
to the group $ \Diff_+(\Ss^n) $ of orientation-preserving diffeomorphisms of $
\Ss^n $.  A homotopy equivalence simply assigns to each $ \bar f $ its Gauss
map.

Roughly, this means that a \tdef{locally convex} immersion $
f\colon \Ss^n \to \Ss^{n+1} $, i.e., one whose principal curvatures are
positive, is uniquely determined, up to homotopy, by the following data: an open
hemisphere containing its image and the Gauss map $ g_f $ of its composition
with central projection.  The antipode of the center of such a hemisphere is
recorded by the last column of $ Q_f $; the remaining columns correspond to a
choice of coordinate axes for the tangent space to $ \Ss^{n+1} $ at this point.
A change of axes affects $ g_f $ accordingly. 

This suggests that $ f \mapsto [Q_f,g_f] $ yields a homotopy equivalence between
the space of locally convex hypersurfaces of $ \Ss^{n+1} $ and the twisted
product $ \SO_{n+2} \times_{\SO_{n+1}} \Diff_+(\Ss^n) $. Recall that the latter
is the quotient of $ \SO_{n+2} \times \Diff_+(\Ss^n) $ under the equivalence
relation which identifies $ (Q,g) $ with $ (QP,P^{-1}g) $ for any $ Q \in
\SO_{n+2} $, $ P \in \SO_{n+1} $ and $ g \in \Diff_+(\Ss^n $).

In fact, a similar weak homotopy equivalence (\tdef{w.h.e.}) holds true for a
more general class of spaces.  However, for these it is simpler to construct the
w.h.e.~in the opposite direction.

\begin{dfn}[$ \sr F(M;I) $]\label{D:F}
	Let $ M^{n+1} $ be a Riemannian manifold and $ I $ be any interval of the
	real line. The set of immersions $ \Ss^n \to M^{n+1} $ whose principal
	curvatures take on values in $ I $, equipped with the $ C^{\infty} $-topology,
	will be denoted by $ \sr F(M;I) $. 

	Suppose $ M $ is a spherical space form and $ J \subs \R \pmod \pi $ is
	an interval of length less than $  \frac{\pi}{2} $.  Then
	it is reasonable to interpret $ \sr F(M;\cot J) $ as the space of \tit{all}
	complete hypersurfaces of $ M $ with principal curvatures in $ \cot J $.
	This is clear when $ M = \Ss^{n+1} $ due to \tref{T:rigidity}. In 
	the general case, one must identify two hypersurfaces which differ by a
	covering map; see \rref{R:meta}.
\end{dfn}

\begin{thm}[homotopical rigidity in $ \Ss^{n+1} $, I]\label{T:homotopical}
	Let $ J \subs (0,\pi) $ be an interval such that $ J \cap (J +
	\frac{\pi}{2}) = \emptyset $.  Let $ r \in J $ be arbitrary and $
	\iota_r \colon \Ss^n \to \Ss^{n+1},~p\mapsto \sin r\,p-\cos r\,e_{n+2} $.
	Then 
	\begin{equation}\label{E:Psir}
		\Psi \colon \SO_{n+2} \times_{\SO_{n+1}} \Diff_+(\Ss^n) \to \sr
		F(\Ss^{n+1};\cot J),\quad [Q,g] \mapsto Q \circ \iota_r \circ g
	\end{equation}
	is a weak homotopy equivalence.
\end{thm}
\begin{rmk}\label{R:know}
	The group $ \Diff_+(\Ss^n) $ is generally disconnected. The number of
	components is finite for all $ n \neq 4 $, and for $ n \geq 5 $ it
	coincides with the order of the group $ \Theta_{n+1} $ of exotic spheres in
	dimension $ n + 1 $ (see \cite{Cerf} and \cite{Milnor4}).  It follows that
	if $ J $ contains a multiple of $ \pi $ in its interior, then $ \sr
	F(\Ss^{n+1};\cot J) $ cannot have the same homotopy type as the twisted
	product. For in this case the former has at least twice as many
	path-components as the latter. More precisely, the number of positive
	principal curvatures separates $ \sr F(\Ss^{n+1};\cot J) $ into closed-open
	subspaces, and the subspace of hypersurfaces with positive (resp.~negative)
	principal curvatures is weakly homotopy equivalent to $ \SO_{n+2}
	\times_{\SO_{n+1}} \Diff_+(\Ss^n) $ by the theorem.  We do not know what the
	homotopy type of the remaining components is. 
\end{rmk}

\begin{uexms}
	If $ n = 2 $ or $ 6 $, then the principal bundle
	\begin{equation*}
		\SO_{n+2} \to \Ss^{n+1} \home \SO_{n+2}/\SO_{n+1} 
	\end{equation*} 
	admits a cross-section $ \sig $.\footnote{A section is readily constructed using
		quaternions and Cayley numbers, respectively; see
	\cite[thm.~8.6]{Steenrod}.}
	This provides an $ \SO_{n+1} $-equivariant homeomorphism $ \Ss^{n+1} \times
	\SO_{n+1} \to \SO_{n+2} $, given by $ (z,P) \mapsto \sig(z)P $. Thus
	\eqref{E:Psir} can be simplified to 
	\begin{equation*}
		\Psi \colon \Ss^{n+1} \times \Diff_+(\Ss^n) \to \sr F(\Ss^{n+1};\cot
		J),\quad (z,g) \mapsto \sig(z) \circ \iota_r \circ g \qquad (n=2,\,6).
	\end{equation*}
	Theorems due to S.~Smale \cite{Smale4} and A.~Hatcher \cite{Hatcher1}
	guarantee that the inclusion $ \SO_{n+1} \inc \Diff_+(\Ss^{n}) $ is a
	homotopy equivalence for $ n=2 $ and $ 3 $, respectively. It follows that in
	these two cases the w.h.e.~\eqref{E:Psir} may be simplified to
	\begin{equation*}
		\Psi \colon \SO_{n+2} \to \sr F(\Ss^{n+1};\cot J),\quad Q \mapsto Q
		\circ \iota_r \qquad (n=2,\,3).\footnote{This assertion is not
		obvious because the homotopy inverse need not be $
		\SO_{n+1} $-equivariant. It is rather a consequence of the fact that the
		space of locally convex immersions in $ \E^{n+1} $ is homotopy equivalent to $
		\Diff_+(\Ss^n) \iso \SO_{n+1}$ ($ n=2,3 $), to
	which the proof of \tref{T:homotopical} is ultimately reduced.}
	\end{equation*}
	It is easily checked that these simplifications
	are equivalent when $ n = 2 $. 
\end{uexms}

Let $ I \subs \R $ be an arbitrary interval, $ M^{n+1} $ an arbitrary 
Riemannian manifold and $ \pr \colon \te{M} \to M $ a Riemannian covering
map. It can be shown \cite[lem.~5.7]{Zuehlke2} that the induced map $
\pr_\ast \colon \sr F(\te{M};I) \to \sr F(M;I),\ f \mapsto \pr \circ f $ is also
a covering map. Moreover, if $ \pr $ is regular, then so is $ \pr_\ast $, and
their automorphism groups are isomorphic via $ \ga \mapsto
\ga_\ast $. Together with \tref{T:homotopical}, this yields the following result; 
for a stronger version, see \rref{R:stronger}.

\begin{thm}[homotopical rigidity in space forms]\label{T:homotopicalspherical}
	Let $ \pr \colon \Ss^{n+1} \to M^{n+1} $ be a Riemannian covering and $ J
	\subs (0,\pi) $ an interval such that $ J \cap (J + \frac{\pi}{2}) =
	\emptyset $. Let $ r \in J $ be arbitrary and $ \iota_r \colon \Ss^n \to
	\Ss^{n+1},~p\mapsto \sin r\,p-\cos r\,e_{n+2} $. Then 
	\begin{equation}\label{E:psistar}
		\pr_\ast \circ \,\Psi \colon \SO_{n+2} \times_{\SO_{n+1}} \Diff_+(\Ss^n)
		\to \sr F(M;\cot J),\quad [Q,g] \mapsto \pr \circ\, Q \circ \iota_r
		\circ g
	\end{equation}
	induces isomorphisms between $ k $-th homotopy groups for $ k \neq 1 $, and
	yields an exact sequence
	\begin{equation*}
		\pushQED{\qed}
		\begin{tikzcd}
		1 \rar &  \pi_1\big(\SO_{n+2} \times_{\SO_{n+1}} \Diff_+(\Ss^n)\big)
		\rar{(\pr_\ast \circ \Psi)_{\ast}} &  
		\pi_1\big(\sr F(M;\cot J)\big) \rar & \pi_1(M) \rar & 1. 
		\end{tikzcd}
		\qedhere 
		\popQED
	\end{equation*}
\end{thm}

\begin{urmk}\label{R:deform}
	The theorem implies in particular that \eqref{E:psistar} induces a bijection
	between the respective sets of path-components. It follows that any
	complete hypersurface $ \Ss^n \to M^{n+1} $ whose principal curvatures are
	constrained to $ \cot (J)  $ can be deformed, through hypersurfaces of the
	same type, to a ``hypersphere'' of $ M $ (i.e., the image of some
	hypersphere of $ \Ss^{n+1} $ under the covering projection defining $ M $).
\end{urmk}

\subsection*{More homotopical rigidity} Another topological rigidity criterion
due to Wang/Xia (\cite[thm.~1.2]{WanXia}) asserts that a complete
hypersurface of $ \Ss^{n+1} $ whose dual has image contained in a 
metric ball of radius $ < \frac{\pi}{2} $ must be diffeomorphic to $ \Ss^n $. We
show that in this case it is an embedding, and apply their result to prove the
following analogue of \tref{T:homotopical} concerning the space $ \sr H $ of all
such hypersurfaces. In the statement $ \tau $ denotes the ``orthogonal
projection'' 
\begin{equation*}
	\tau \colon \Ss^{n+1} \ssm \se{\pm e_{n+2}} \to \Ss^n, \quad p \mapsto
	\frac{p-\gen{p,e_{n+2}}e_{n+2}}{\abs{p-\gen{p,e_{n+2}}e_{n+2}}} .
\end{equation*}

\begin{thm}[homotopical rigidity II]\label{T:homotopical2}
	Let $ \iota \colon \Ss^n \inc \Ss^{n+1} $ denote set inclusion and $ \sr
	H $ be the space described above. Then:
	\begin{equation}\label{E:barPsi}
		\bar\Psi \colon \SO_{n+2} \times_{\SO_{n+1}} \Diff_+(\Ss^n) \to \sr H,
		\qquad [Q,g]   \mapsto Q \circ \iota \circ g
	\end{equation}
	is a homotopy equivalence. In fact, given $ f \in \sr H $, let $ c_f \in
	\Ss^{n+1} $ be the center of the image of its dual. Choose any $ Q_f
	\in \SO_{n+2} $ with $ Q_f(-e_{n+2}) = c_f $ and set $ g_f = \tau \circ
	Q_f^{-1} \circ f $. A homotopy inverse of \eqref{E:barPsi} is defined by:
	\begin{alignat*}{9}
		 \bar\Phi \colon \sr H \to \SO_{n+2} \times_{\SO_{n+1}} \Diff_+(\Ss^n),
		 \qquad f   \mapsto [Q_f,g_f].
	\end{alignat*}
\end{thm}

The analogues for $ n=1 $ of the spaces studied here are more
complicated. A complete description in the locally convex case is obtained in
\cite{Saldanha3}, and \cite{SalZueh} contains partial results in the general case.

\subsection*{Outline of the sections}
In \S\ref{S:topological} we compute the principal curvatures of the dual and of
the normal translates of an arbitrary immersion $ f \colon N^n \to \Ss^{n+1} $
in terms of those of $ f $. This is combined with \tref{T:doCarmo/Warner}
to establish the topological rigidity theorem \tref{T:rigidity} for
hypersurfaces of $ \Ss^{n+1} $. From this we then derive the
aforementioned related results of Wang/Xia and Longa/Ripoll. It is also shown
that if $ f $ is locally convex, then there exists an open hemisphere which
contains the images of both $ f $ and $ f^{\star} $.

The purpose of \S\ref{S:homotopical} is to prove the homotopical rigidity theorem
\tref{T:homotopical}. First it is proven that the inclusion $ \sr
F(\Ss^{n+1};\cot J) \inc \sr F(\Ss^{n+1};\cot J') $ is a weak homotopy
equivalence for any pair of intervals $ J \subs J' \subs (0,\pi) $ with $
J' $ disjoint from $ J' + \frac{\pi}{2} $. The proof relies on certain
M\"obius transformations and the results of \S 1. Using this,
\tref{T:homotopical} is reduced to the locally convex case, which is proven
directly by closely following the sketch provided above.

In \S\ref{S:hemispherical}, \tref{T:homotopical2} is established.  Even
though its statement is similar to that of \tref{T:homotopical}, its proof is
considerably easier.  

Finally, \S\ref{S:spherical} contains a proof of \tref{T:spherical}, 
a resulting interpretation of $ \sr F(M;I) $ as a space of what we call
``irreducible hypersufaces'', and a slightly stronger but more
technical formulation of \tref{T:homotopicalspherical}. At the end
of the paper three related open problems are listed.

\section{Topological rigidity}\label{S:topological}

We begin by studying the properties of translates of a
hypersurface in the direction of its Gauss map.

\begin{lem}[parallel immersions]\label{L:fr}
	Let $ f \colon N^n \to \Ss^{n+1} $ be an immersion and $ \nu_f $ its Gauss
	map. Given $ r \in \R $, define $ f_r \colon N^n \to \Ss^{n+1} $ by
	\begin{equation}\label{E:fr}
		f_r(p) = \exp_{f(p)} (r\nu_f(p)) = \cos r\,f(p) + \sin r \, \nu_f(p)
		\qquad (p \in N).
	\end{equation}
	Then the following assertions hold: 
	\begin{enumerate}
		\item [(a)] $ f_r $ is an immersion if and only if $ r \pmod \pi $ is
			\emph{not} a principal radius of $ f $.
		\item [(b)] Suppose that $ f_r $ is an immersion. Let $ p \in N $ 
			and let $ l $ denote the number of principal radii $
			\rho \in (0,\pi) \pmod \pi $ of $ f $ at
			$ p $ such that
			\begin{equation}\label{E:l}
				\sin \rho\, \sin(\rho - r) < 0.
			\end{equation}
			Then $ l $ is independent of $ p $. Moreover, $ u \in TN $ is a
			principal direction for $ f $ associated to the principal radius $
			\rho \pmod \pi $ if and only if $ u $ is a principal
			direction for $ f_r $ associated to the principal radius $
			(-1)^{l}(\rho - r) \pmod \pi $.
		\item [(c)] If $ f_r $ is an immersion and $ r' \in \R $, then $
			f_{r+r'} = (f_{r})_{(-1)^{l}r'} $.
	\end{enumerate}
\end{lem}

\begin{urmk}
	Note that the expression in \eqref{E:l} is independent of the representative
	$ \rho $ of $ \rho \pmod \pi $, as expected, but its sign depends on the
	representative $ r $ of $ r \pmod \pi $.
\end{urmk}

\begin{proof}
	Let $ u $ be a principal direction for $ f $ associated
	to the principal radius $ \rho \in (0,\pi) \pmod \pi $. Then
	\begin{equation}\label{E:dfr}
		d(f_r)(u) = (\cos r - \cot \rho \sin r)df(u)=\frac{\sin (\rho - r)}{\sin
		\rho}df(u).
	\end{equation}
	This is a positive or negative multiple of $ df(u) $ according to whether 
	$ \sin \rho \sin (\rho - r) $ is positive or negative. In particular, it is
	a nonzero multiple if and only if $ r \not \equiv \rho \pmod \pi $, proving
	(a).
	
	Suppose now that $ f_r $ is an immersion. The number $ l $ defined in (b) is
	independent of the chosen point $ p \in N $ by connectedness of $ N $. 
	It follows directly from \eqref{E:dfr} and our definition \dref{D:Gauss} of
	the Gauss map that 
	\begin{equation}\label{E:nufr}
		\nu_{f_r} = (-1)^l(\cos r\,\nu_f -\sin r\, f ),
	\end{equation}
	so that
	\begin{equation*}\label{E:dnufr}
		d(\nu_{f_r})(u) = (-1)^{l+1} (\sin r + \cos r \cot \rho)df(u).
	\end{equation*}	
	Combining this with \eqref{E:dfr}, one deduces that
	\begin{equation*}
		-d(\nu_{f_r})(u) = (-1)^l \cot(\rho - r)d(f_r)(u).
	\end{equation*}
	Thus, $ u $ is a principal direction for $ f_r $ associated to the
	principal radius $ (-1)^l (\rho - r) \pmod \pi $. The converse is obtained
	by applying the same argument to $ f_r $ using the identity $ f =
	(f_r)_{(-1)^{l+1} r} $, which is derived from (c). In turn, part (c) follows
	from a straightforward computation using \eqref{E:nufr}.
\end{proof}

\begin{crl}[curvature of the dual]\label{C:dual}
	Let $ M^{n+1} $ be a spherical space form and let
	$ f \colon N^n \to M^{n+1} $ be an immersion. Then its dual $ f^{\star}
	\colon N \to M $ is an immersion if and only if $ 0 $ is \emph{not} a
	principal curvature of $ f $.  Moreover, in this case:
	\begin{enumerate}
		\item [(a)] The dual of $ f^{\star} $ is $ (-1)^{l+1} f $, where $ l $
			is the number of positive principal curvatures of $ f $ at any point
			of $ N $, and by definition $ +f = f $ and  $ -f $ is given by $ p \mapsto
			\exp_{f(p)}\big(\pi \nu_f(p)\big) $.
		\item [(b)]	$ u $ is a principal direction for $ f $
			associated to the principal curvature $ \ka $ if and
			only if $ u $ is a principal direction for $ f^{\star} $ associated to
			the principal curvature $ (-1)^{l+1} \ka^{-1} $. 
		\item [(c)]	$ J(f^{\star}) \equiv \frac{\pi}{2} \pm J(f) \pmod \pi $.
	\end{enumerate} 
\end{crl}
\begin{proof}
	For $ M = \Ss^{n+1} $, apply \lref{L:fr} and \eqref{E:nufr} in the case
	where $ r = \frac{\pi}{2} $. The general case is reduced to this one by
	lifting $ f $ to $ \te{f} \colon \te{N} \to \Ss^{n+1}
	$, where $ \te{N} $ is the universal cover of $ N $, and noting that the
	dual of $ \te{f} $ is a lift of the dual of $ f $. (For a careful argument
		establishing the latter, see the third \hyperlink{paragraph}{paragraph}
	of the proof of \tref{T:spherical}).
\end{proof}

\begin{lem}\label{L:convex}
	Let $ N^n $ be closed and $ f \colon N^n \to \Ss^{n+1} $ be an 
	immersion. Define
	\begin{equation*}
		F \colon N^n \times \big[0,\tfrac{\pi}{2}\big] \to \Ss^{n+1}\quad
		\text{by} \quad (p,t) \mapsto \exp_{f(p)}\big( t\nu(p) \big) =
		\cos t\, f(p) + \sin t\, f^\star(p).
	\end{equation*}
	If the principal curvatures of $ f $ are all positive (resp.~nonnegative),
	then the open (resp.~closed) hemisphere determined by the center of $ f(N) $
	contains the image of $ F $.
\end{lem}

\begin{rmk}\label{R:convex}
	The center of the osculating circle of a normal section to $ f $ at $ p $ is $
	\exp_p(\rho \nu(p)) $, where $ \rho $ is the radius of curvature of the
	normal section. Thus, a convenient geometric interpretation for $ F $ is the
	following: 	The image of $ N \times \big( 0,\frac{\pi}{2} \big) $ under $ F
	$ is the locus of all possible centers of osculating circles to normal
	sections of $ f $, given the information that the principal curvatures of $
	f $ are positive but otherwise unrestricted. 
\end{rmk}

\begin{proof}
	Assume first that the principal curvatures are positive. Then $ f(N) $ is the
	boundary of a strictly convex body $ B $; cf.~ \tref{T:doCarmo/Warner} and
	\cite[lem.~2.2]{CarWar}. Therefore, $ f(N) \ssm \se{f(q)} $ is contained in
	the open hemisphere determined by $ \nu(q) $ for each $ q \in N $.
	Equivalently, 
	\begin{equation}\label{E:fnuf}
		\gen{f(p),f^\star(q)} \geq 0\text{\quad for all $ p,\,q \in N $},
	\end{equation}
	with equality holding if and only if $ p = q $. Let $ c \in \Ss^{n+1} $
	denote the center of $ f(N) $. By convexity, $ c $ lies in the interior of $
	B $, so it may be written as  $ c = a_1f(p_1)+a_2f(p_2) $ for
	some $ a_1,\,a_2 > 0 $ and $ p_1 \neq p_2 \in N $.  It follows from
	\eqref{E:fnuf} that $ \gen{c,f^{\star}(q)} > 0 $ for all $ q \in N $.
	Consequently, the open hemisphere determined by $ c $ contains both $
	f^{\star}(N) $ and $ f(N) $, the latter by the definition of $ c $. But
	being convex, this hemisphere must also contain the image of $ F $.  
	
	In case the principal curvatures are only nonnegative, the assertion can be
	deduced from the preceding paragraph by regarding $ f $ as the limit of $
	f_r $ as $ r \to 0^+ $. 
\end{proof}

\begin{proof}[Proof of \pref{T:rigidity}]
	Suppose first that $ J(f) $ does not contain $ 0 \pmod \pi $. Then by
	\rref{R:principalcurvatures}\,(i) combined with the theorem of Bonnet-Myers,
	$ N $ must be closed.  Together with the hypothesis that $ \length(J(f)) <
	\frac{\pi}{2} $, this in turn implies that $ J(f) $ is contained in 
	$ (r,r + \frac{\pi}{2}) \pmod \pi $ for some $ r \in (0,\frac{\pi}{2}) $.
	By \lref{L:fr}\,(b),
	\begin{equation*}
		J(f_r) \equiv J(f) - r \subs (0,\tfrac{\pi}{2}).
	\end{equation*}
	In particular, \tref{T:doCarmo/Warner} applied to $ f_r $ guarantees that $
	N $ is diffeomorphic to $ \Ss^n $, thus establishing (a) in this case.
	Moreover, by \lref{L:fr}\,(c), $ f = (f_r)_{-r} $. Thus, to prove (b) it
	will be sufficient to show that: if $ g \colon N^n \to \Ss^{n+1} $ is
	an immersion with $ J(g) \subs (0,\frac{\pi}{2}) \pmod \pi $, then $ g_{-r}
	$ is an embedding for all $ r \in [0,\frac{\pi}{2}] $. 

	Let $ B $ be the convex body bounded by $ g(N) $. Because the
	principal curvatures of $ g $ are positive (instead of negative), 
	\begin{equation}\label{E:center}
		\text{$ \nu_g $ points towards the interior of $ B $ at every
		point of $ N $.}
	\end{equation}
	Define
	\begin{equation*}
		s = \sup \set{t \in [0,\tfrac{\pi}{2}]}{g_{-r} \text{ is an embedding
		for all $ r \in [0,t] $}}.
	\end{equation*}
	Suppose for the sake of obtaining a contradiction that $ s < \frac{\pi}{2}
	$. Recall that embeddings form an open subset of $ C^{\infty}(N,\Ss^{n+1}) $.
	Moreover, $ g_{-r} $ is an immersion for all $ r \in [0,\frac{\pi}{2}]
	$ by \lref{L:fr}\,(a). Hence $ g_{-s} $ cannot be injective.  Let $ p\neq q
	\in N $ be such that $ g_{-s}(p) = g_{-s}(q) $.  As we are in codimension 1,
	$ d(g_{-s})(TN_p) $ and $ d(g_{-s})(TN_q) $ are not transverse, for
	otherwise there would exist $ s'<s $ such that $ g_{-s'} $ is likewise not
	injective. Thus $ \nu_{g_{-s}}(p) = \pm \nu_{g_{-s}}(q) $ in addition to $
	g_{-s}(p) = g_{-s}(q) $. 
	If $ \nu_{g_{-s}}(p) = \nu_{g_{-s}}(q) $, then it
	follows immediately from the definition that
	\begin{equation*}
		g(p) = (g_{-s})_{s}(p) = (g_{-s})_{s}(q) = g(q);
	\end{equation*}
	but this is impossible because $ g $ is an embedding, by \tref{T:doCarmo/Warner}.
	
	Hence $ \nu_{g_{-s}}(p) = -\nu_{g_{-s}}(q) $. Let $ C $ be the unique
	great circle tangent to these.  Then $ g(p) $ and $ g(q) $ bound a segment $
	S \subs B $ of $ C $ of length less than $ \pi $, since $ g(N) $ is
	contained in an open hemisphere. But by \eqref{E:center}, $
	\nu_{g}(p) $ and $ \nu_{g}(q) $ both point to the interior of $ S $.  Hence
	$ g_{-s}(p) $ and $ g_{-s}(q) $, which are obtained from $ g(p) $ and $ g(q)
	$ by moving a distance $ s < \frac{\pi}{2} $ towards the exterior of $ S $
	along $ C $, cannot be equal. This is a contradiction.  Thus, if $ J(f) $
	does not contain $ 0 \pmod \pi $, then $ N $ is diffeomorphic to $ \Ss^n $
	and $ f $ is an embedding.

	\label{P:paragraph}Suppose now that $ J(f) $ does not contain $
	\frac{\pi}{2} \pmod \pi $.  Then, by \cref{C:dual}\,(c), $ J(f^{\star})
	\equiv \tfrac{\pi}{2} \pm J(f) \pmod \pi $ does not contain $ 0 \pmod \pi $.
	Let $ g $ and $ g^\star $ denote the Riemannian metrics on $ N $ induced by
	$ f $ and $ f^\star $, respectively. By \rref{R:principalcurvatures}\,(ii)
	there exists $ k > 0 $ such that the principal curvatures of $ f $ are
	greater than $ k $ in absolute value.  Using \eqref{E:dfr} with $ r =
	\frac{\pi}{2} $, it follows that
	\begin{equation}\label{E:gstar}
		g^\star (u,u) \geq k^2 \,g(u,u)
	\end{equation}
	for any principal direction $ u $. But by \cref{C:dual}\,(b), the principal
	directions of $ f $ and $ f^\star $ are the same. Since at every point it is
	possible to find an orthogonal basis (with respect to either metric)
	consisting of such directions, \eqref{E:gstar} holds for any tangent vector
	$ u $.  It follows that $ (N,g^\star) $ is complete since $ (N,g) $ is.
	Therefore $ f^\star $ satisfies the hypotheses of part (b). By what was
	proven above, it must be an embedding and $ N $ must be diffeomorphic to $
	\Ss^n $. 
\end{proof}

\begin{rmk}\label{R:lastparagraph}
	The following fact established in the preceding paragraph 
	will be invoked later: Under the hypotheses of \tref{T:rigidity}\,(c), $ N $
	is complete with respect to the metric induced by $ f^\star $.
\end{rmk}

\begin{crl}
	Let $ f \colon N^n \to \Ss^{n+1} $ be a closed hypersurface whose sectional
	curvatures are all greater than 1. Then $ f $ is homotopic, through
	embeddings, to either $ f^\ast $ or $ -f^\ast $.
\end{crl}
\begin{proof}
	The homotopy is provided by $ t \mapsto f_t $ for $ t \in
	\big[0,\frac{\pi}{2}\big] $ in case the principal curvatures are all
	negative, and by $ t \mapsto f_{-t} $ otherwise.
\end{proof}

\begin{thm}[Wang/Xia, {\cite[thm.~1.1]{WanXia}}]\label{T:Wang/Xia}
	Let $ N^n $ be a closed manifold and $ f \colon N^n \to \Ss^{n+1} $ be an
	immersion.  Suppose that $ f(N) $ is contained in an open hemisphere and
	that the Gauss-Kronecker curvature of $ f $ does not vanish.
	Then $ N $ is diffeomorphic to $ \Ss^n $.
\end{thm}

\begin{proof}
	Comparison of $ f(N) $ with its metric circumsphere
	(cf.~\cite[II.2.7]{BriHae}), shows that $ N $ contains a point where the
	principal curvatures  have the same sign. Since by hypothesis the
	Gauss-Kronecker curvature is nonvanishing and $ N $ is connected, all
	principal curvatures must have the same sign. Thus $ J(f) \subs \big(
	0,\frac{\pi}{2} \big) $ or $ \big( \frac{\pi}{2},\pi \big) \pmod \pi $, and
	\tref{T:rigidity}\,(a,\,b,\,c) apply. 
\end{proof}

\begin{thm}[Longa/Ripoll, {\cite[thm.~1]{LonRip1}}]\label{T:Longa/Ripoll}
	Let $ N^n $ be closed and $ f \colon N^n \to \Ss^{n+1} $ be an 
	immersion. Let $ r $ be the radius of the smallest closed metric ball
	containing $ f(N) $. If each principal curvature $ \ka $ of $ f $ satisfies
	$ \abs{\ka} > \tan \big(\frac{r}{2}\big) $,	then $ N $ is diffeomorphic to $
	\Ss^n $.
\end{thm}
\begin{proof}
	If $ r < \frac{\pi}{2} $, then $ f(N) $ is contained in an open hemisphere
	and no principal curvature is zero, hence \tref{T:Wang/Xia} applies.  If $ r
	\geq \frac{\pi}{2} $, then all principal curvatures are greater than 1 in
	absolute value.  Equivalently, $ J(f) $ is contained in $ (
	-\frac{\pi}{4},\frac{\pi}{4} ) \pmod \pi $, so \tref{T:rigidity}\,(a,\,c)
	apply.
\end{proof}

\section{Homotopical rigidity}\label{S:homotopical}

The following elementary result will be used together with \lref{L:convex} to
uniformly increase the principal curvatures of a locally convex immersion.
\begin{lem}\label{L:Moebius}
	Let  $ C \subs \Ss^{n+1} $ be a circle of radius $ < \frac{\pi}{2} $. Let
	$ \sig_c \colon \Ss^{n+1} \to \hat\R^{n+1} = \R^{n+1} \cup \se{\infty} $
	denote stereographic projection from $ -c $ and $ M_s $ \tup($ s \in (0,1]
	$\tup) be the family of M\"obius transformations defined by
	\begin{equation}\label{E:Moebius}
		M_s \colon \Ss^{n+1} \to \Ss^{n+1}, \quad p \mapsto
		\sig_c^{-1}(s\sig_c(p)).
	\end{equation}
	Then the radius of $ M_s(C) $ decreases strictly (and converges to 0) as $
	s $ decreases to 0 if and only if the center of $ C $ lies in the closed
	hemisphere determined by $ c $.
\end{lem}
\begin{proof}
	No generality is lost in assuming that $ c = e_{n+2} $. Recall that
	M\"obius transformations map circles to circles. The pull-back of
	the round metric under $ \sig^{-1} = \sig_{c}^{-1} $ is given by
	\begin{equation*}
	\qquad	g_x = 4\big(1+\abs{x}^{2}\big)^{-2}\,\bar g_x\qquad (x \in
		\R^{n+1}),
	\end{equation*}
	where $ \bar g $ is the euclidean metric of $ \R^{n+1} $.  Let $
	[p,q] $ be a diameter of $ C $ whose extension contains $ e_{n+2} $ and,
	without loss of generality, let $ I \subs \hat{\R} = \hat{\R} e_1 $ be
	its image under $ \sig $. 
	\begin{enumerate}
		\item [$\ast$] If $ \infty \in I $, then $ C $ is contained in
			the open southern hemisphere.  For $ s \in (0,1] $ close to 1, the
			length of $ [p,q] = [p,-e_{n+2}] \cup [-e_{n+2},q] $ increases
			strictly as $ s $ decreases.
		\item [$ \ast $]  If $ 0 \in I $, then $ C $ is contained in the open
			northern hemisphere, and $ I = [a,b] $ with $ ab \leq 0 $.  The
			length of $ [sa,sb] $ with respect to $ g $ decreases strictly with
			$ s \in (0,1] $ since it is the sum of those of $ [sa,0] $ and $
			[0,sb] $.
		\item [$ \ast $] If neither $ \infty $ nor $ 0 $ lie in $ I $, then $ I
			= [a,b] $ with $ ab > 0 $. A simple computation shows that the
			length of $ [sa,sb] $ decreases strictly together with $ s $ if and
			only if $ ab \leq 1 $.  As reflection in the equator of $ \Ss^{n+1}
			$ corresponds under stereographic projection to inversion of $
			\hat\R^{n+1} $ in $ \Ss^n $, the latter condition is equivalent to
			the midpoint of $ [p,q] $ (i.e., the center of $ C $) lying inside
			the closed northern hemisphere.
	\end{enumerate}
	The conclusion follows since $ M_s([p,q]) $ is a diameter
	of $ M_s(C) $. 
\end{proof}

\begin{crl}\label{C:Moebius}
	Let $ N^n $ be closed, $ f \colon N^n \to \Ss^{n+1} $ be a locally convex
	immersion and $ c $ be the center of $ f(N) $. For $ M_s $ as in
	\eqref{E:Moebius}, let $ \mu(s) $ be the smallest principal curvature of $
	M_s \circ f $.  Then $ \mu(s) $ increases strictly to $ +\infty $ as $ s \in
	(0,1] $ decreases to 0.  
\end{crl}
\begin{proof}
	Let $ p \in N $ be arbitrary and $ u \in TN_p $. Let $ \eta $ be the
	corresponding normal section, i.e., the curve which is the intersection of 
	$ f(V) $ with the unique totally geodesic 2-sphere $ \Sig $ tangent to $ u $
	and $ \nu(p) $,  where $ V $ is a small neighborhood of $ p $. Finally, let
	$ C \subs \Sig $ be the unique osculating circle to $ \eta $ at $ f(p) $.
	Then $ M_s(C) $ is the osculating circle to the normal
	section of $ M_s \circ f $ determined by $ u $. By \lref{L:convex} and
	\rref{R:convex}, the center of $ C $ lies inside the hemisphere determined
	by $ c $. Hence the radius of $ M_s(C) $ decreases strictly to 0 with $ s \in
	(0,1] $ by \lref{L:Moebius}. This implies the conclusion of the lemma.
\end{proof}

\begin{prp}\label{P:whe}
	Let $ J \subs J' \subs (0,\pi) $ be nondegenerate intervals with $ J' \cap
	(J' + \frac{\pi}{2}) =\emptyset $.  Then the inclusion
	\begin{equation}\label{E:JJ'}
		i\colon \sr F(\Ss^{n+1};\cot J) \inc \sr F(\Ss^{n+1};\cot J')
	\end{equation}
	is a weak homotopy equivalence.
\end{prp}
\begin{proof}
	The proof will be broken into several steps.

	\ssk \noindent \tit{Step 1:} If $ J = (a,b) $ and $ J'=[a,b) $ or $ (a,b]
	\subs (0,\pi) $ (with $ b-a \leq \frac{\pi}{2}) $, then the inclusion
	\begin{equation*}
		\sr F\big(\Ss^{n+1};\cot J \big) \inc \sr F\big(\Ss^{n+1};\cot
		J' \big) 
	\end{equation*}
	is a weak homotopy equivalence. \ssk

	By compactness of $ \D^k $,  given any map of pairs
	\begin{equation*}
		\qquad F \colon \big(\D^k,\Ss^{k-1}\big) \to 
		\big(\sr F\big(\Ss^{n+1};\cot (a,b]\big)\,,\,\sr
		F\big(\Ss^{n+1};\cot(a,b)\big)\big)\qquad (k \in \N^+), 
	\end{equation*}
	there exists $ \eps > 0 $ such that the image of $ F $ is actually contained
	in $ \sr F\big(\Ss^{n+1};\cot (a+\eps,b]\big) $.  Consequently, by
	\lref{L:fr}\,(b), $ H \colon (s,p) \mapsto F(p)_{s\eps} $ defines a homotopy
	of pairs between $ F=H_0 $ and a map $ H_1 $ with image in 
	\begin{equation*}
		\sr F\big(\Ss^{n+1};\cot (a,b-\eps]\big) \subs 
		\sr F\big(\Ss^{n+1};\cot (a,b)\big).
	\end{equation*}
	This establishes the triviality of the relevant relative homotopy groups.
	The proof for $ [a,b) $ is analogous.

	\ssk \noindent \tit{Step 2:} It can be assumed that $ J'$ is open. \ssk

	Indeed, suppose that we have proved the proposition in this special case.
	Let $ J \subs J' $ be as in the hypothesis, but with $ J' $ not
	necessarily open.
	
	If $ \length(J') < \frac{\pi}{2} $, then we can find an open interval
	$ J'' $ with $ J \subs J' \subs J'' $. Under our assumption the latter
	inclusion, as well as the inclusion $ J \subs J'' $, yield w.h.e.~of the
	corresponding spaces, hence so does $ J \subs J' $. 
	
	If $ \length(J') = \frac{\pi}{2} $, then $ J' $ cannot be closed, since $ J'
	\cap (J'+\frac{\pi}{2}) = \emptyset $. If it is open, there is nothing to
	prove. Assume then that $ J' $ is half-open, and let $ I,\,I' $ denote the
	interiors of $ J,\,J' $.  Then in the following commutative diagram of
	inclusions:
	\begin{equation*}
	\begin{tikzcd}
		\sr F(\Ss^{n+1};\cot J) \arrow[hookrightarrow]{r}{} & \sr
		F(\Ss^{n+1};\cot J') \\ 
		\sr F(\Ss^{n+1};\cot I) \arrow[hookrightarrow]{r}{}
		\arrow[hookrightarrow]{u}{} &  \sr F(\Ss^{n+1};\cot I')
		\arrow[hookrightarrow]{u}{}
	\end{tikzcd}
	\end{equation*}
	the bottom arrow is a w.h.e.~as $ I' $ is open, and the arrow on the
	right is a w.h.e.~by step 1. If the length of $ J $ is also $ \frac{\pi}{2}
	$, then either $ I=J $ or step 1 applies to the left vertical arrow; in
	either case, it is a w.h.e.. Finally, if the length of $ J $ is less than $
	\frac{\pi}{2} $, then the left arrow is a w.h.e.~by the preceding paragraph
	(with $ I $ in place of $ J $ and $ J $ in place of $ J' $). Consequently,
	the top inclusion is also a w.h.e..

	\ssk \noindent \tit{Step 3:} It can be assumed that $ J' \subs
	(0,\frac{\pi}{2}) $. \ssk

	By step 2, no generality is lost in assuming that $ J' $ is open, say, $ J'
	= (a,b) $.  Let $ \rho = \frac{a+b}{2} - \frac{\pi}{4} $.  By
	\lref{L:fr}\,(b), $ f \mapsto f_\rho $ yields a homeomorphism of pairs
	\begin{equation*}
		\big(  \sr F\big(\Ss^{n+1};\cot J'\big)\,,\, \sr F\big(\Ss^{n+1};\cot
		J\big) \big) \to 
		\big(  \sr F\big(\Ss^{n+1};\cot (J'-\rho)\big)\,,\, \sr
		F\big(\Ss^{n+1};\cot (J-\rho)\big) \big).
	\end{equation*}
	But $ J'-\rho \subs (0,\frac{\pi}{2}) $	by the hypothesis on $ J' $ and the
	choice of $ \rho $. 

	\ssk \noindent \tit{Step 4:} It can be assumed that $ J' = (0,\frac{\pi}{2})
	$. \ssk

	For, suppose that $ J \subs  J' \subs (0,\frac{\pi}{2}) $ as in step 3.
	To prove that \eqref{E:JJ'} is a w.h.e.~, it is sufficient to prove that the
	inclusions $ J \subs (0,\frac{\pi}{2}) $ and $ J' \subs (0,\frac{\pi}{2}) $
	induce w.h.e.~of the corresponding spaces, that is, it can be assumed that
	the interval defining the target is $ (0,\frac{\pi}{2}) $, as claimed.
	
	\ssk \noindent \tit{Step 5:} The proposition holds in case $ J \subs
	J'=(0,\frac{\pi}{2}) $ is not closed. \ssk

	By step 1, if $ J $ is half-open then it can be replaced with its interior;
	in other words, it can be assumed that $ J = (a,b) \subs \big(
	0,\frac{\pi}{2} \big) = J' $.  Let $ J_t = (a-ta,b-ta) $ and define
	\begin{equation*}
		h_t \colon \sr F(\Ss^{n+1};\cot J) \to \sr F(\Ss^{n+1};\cot J_t),\quad
		f\mapsto f_t \qquad (t \in [0,a]).
	\end{equation*}
	By \lref{L:fr}, each $ h_t $ is a homeomorphism. Let
	\begin{equation*}
		i_{J_t} \colon \sr F(\Ss^{n+1};\cot J_t) \inc \sr
		F(\Ss^{n+1};(0,+\infty)).
	\end{equation*}
	Then $ i_{J_0} = i_{J_0} \circ h_0 \iso i_{J_t} \circ h_t $.	Therefore,
	to show that $ i_J = i_{J_0} $ is a weak homotopy equivalence, it is
	sufficient to prove that $ i_{J_1} $ is one. For this, let 
	\begin{equation*}
		\qquad F \colon \big(\D^k,\Ss^{k-1}\big) \to \big(\sr
		F\big(\Ss^{n+1};(0,+\infty)\big)\,,\,\sr
		F\big(\Ss^{n+1};\cot (0,b-a)\big)\big) \qquad (k \in \N^+)
	\end{equation*}
	be continuous and set $ f^z = F(z) $ ($ z \in \D^k $). Let $ c(z) $ denote
	the center of $ f^z(\Ss^n) $. Define a homotopy
	\begin{equation*}
		H \colon (0,1] \times \D^k \to \sr F\big(\Ss^{n+1};(0,+\infty)\big)\quad
		\text{by}\quad H(s,z) = M_s^z \circ f^z,
	\end{equation*}
	where $ M_s^z $ is given by \eqref{E:Moebius} with $ c=c(z) $.  By
	\cref{C:Moebius}, the restriction of $ H $ to $ [\eps,1] \times \D^k $ is a
	homotopy of pairs, and for $ \eps > 0 $ sufficiently small, $ H_\eps(\D^k)
	\subs \sr F\big(\Ss^{n+1};\cot(0,b-a)\big) $. This establishes that $
	i_{J_1} $ is a weak homotopy equivalence.

	\ssk \noindent \tit{Step 6:} The proposition holds in case $ J \subs J' =
	\big( 0,\frac{\pi}{2} \big) $ is closed. \ssk

	Let $ J = [a,b] \subs (0,\frac{\pi}{2}) $. By step 5, it suffices to show
	that 
	\begin{equation*}
		\sr F\big(\Ss^{n+1};\cot [a,b)\big) \inc 
		\sr F\big(\Ss^{n+1};\cot [a,b]\big)
	\end{equation*}
	is a weak homotopy equivalence. Let
	\begin{equation*}
		\qquad F \colon \big(\D^k,\Ss^{k-1}\big) \to 
		\big(\sr F\big(\Ss^{n+1};\cot [a,b]\big)\,,\,
		\sr F\big(\Ss^{n+1};\cot [a,b)\big)\big)\qquad (k \in \N^+)
	\end{equation*}
	be continuous. Fix $ s \in (0,1] $ for now and let $ f^z = F(z) $ and $
	M^z_s $ be as above. Recalling the construction of normal translates in
	\eqref{E:fr}, define
	\begin{equation}\label{E:H}
		H_s \colon [0,a) \times \D^k \to \sr F\big(\Ss^{n+1};(0,+\infty)\big)
		\quad \text{by} \quad H_s(t,z) = (M^z_s \circ f^z_{t})_{-t}.
	\end{equation}
	We would like to extend this definition to $ [0,a] \times \D^k $; the
	problem is that, by \lref{L:fr}\,(a), $ f^z_a $ is not immersive at
	points where some principal radius equals $ a$. Intuitively, at such a point
	the corresponding principal curvature is $ +\infty $. However, this can be
	circumvented. 

	By \eqref{E:nufr}, the Gauss map of $ f^z_t $ is given by
	\begin{equation}\label{E:eq1}
		\nu_{f^z_t} = \cos t \, \nu_{f^z} - \sin t \, f^z \qquad (t
		\in [0,a),~z \in \D^k).
	\end{equation}
	Further, since M\"obius transformations are conformal, the Gauss map of $
	M_s^z \circ f^z_t $ is given by
	\begin{equation}\label{E:eq2}
		\nu_{M_s^z \circ f^z_t}(p) =
		\frac{d(M_s^z)_{f^z_t(p)}(\nu_{f^z_t}(p))}{\abs{\text{numerator}}}
		\qquad (t \in [0,a),~z \in \D^k,~p \in \Ss^n).
	\end{equation}
	Even though strictly speaking $ f^z_a $ need not have a Gauss map, the
	expression on the right side of \eqref{E:eq1} is still sensible (and 
	smooth) for $ t = a $. Hence, so is that on the right side of
	\eqref{E:eq2}; note that the denominator does not vanish since $ M_s^z $ is
	a diffeomorphism.  It is thus possible to extend \eqref{E:H} to all of $
	[0,a] \times \D^k $ by comparing \eqref{E:fr} and setting
	\begin{equation*}
		H_s(t,z) =  \cos t\, (M_s^z \circ f^z_t) - \sin t\,(\nu_{M^z_s\circ
		f^z_t})  \qquad (t \in [0,a],~z \in \D^k).
	\end{equation*}
	Now $ M_1^z = \id_{\Ss^{n+1}} $, hence $ H_1(t,z) = f^z $
	for all $ (t,z) \in [0,a] \times \D^k $.  Therefore, by compactness of
	the latter, there exists $ \eps>0 $ such that $ H_s(t,z) $ is 
	an immersion for all $ (s,t,z) \in [1-\eps,1] \times [0,a] \times
	\D^k $.

	Let $ s \in [1-\eps,1] $ and $ z \in \D^k $ be arbitrary. Since the
	principal radii of $ f_t^z $ take on values in $ [a-t,b-t] \subs (0,b-t] $ for
	all $ t \in [0,a) $ by hypothesis, those of $ H_s(t,z) $ lie in
	$ (t,b] $ by \cref{C:Moebius}. Hence, those of $ H_s(a,z) $ take on values
	in $ [a,b] $ by continuity. We claim that they actually take values in $
	[a,b) $ for $ s < 1 $. 

	Given a curve $ \ga $ on $ \Ss^2 $, define $ \ga_a $ by $ \ga_a(\tau)
	= \cos a\, \ga(\tau) + \sin a\, \nu_{\ga}(\tau) $, where $ \nu_\ga $
	is its normal unit vector. Note that if $ \ga $ parametrizes a
	circle of radius $ r > a $, then $ \ga_a $ parametrizes a circle of
	radius $ r - a $. Let $ C $ be the osculating circle to the normal
	section of $ f^z $ determined by an arbitrary tangent vector $ u \in T\Ss^n $.
	Then $ \te{C} = \big(M_s(C)\big)_{-a} $ is the osculating circle to
	the normal section of $ H_s(a,z) $ determined by $ u $. If the
	radius of $ C $ is greater than $ a $ and $ s < 1 $, then the radius of $
	\te{C} $ is smaller than that of $ C $ by \lref{L:Moebius}, proving the
	claim.

	We conclude that $ G\colon (s,z) \mapsto H_s(a,z) $ defines a homotopy of
	pairs connecting $ G_1 = F $ and a map $ G_{1-\eps} $ with $
	G_{1-\eps}(\D^k) \subs \sr F\big( \Ss^{n+1};\cot [a,b) \big) $. This
	completes the proof of step 6.

	The conclusion now follows from the combination of steps 4--6.
\end{proof}

\begin{prp}[homotopical rigidity of locally convex immersions]\label{P:positive}
	Let $ r \in (0,\frac{\pi}{2}) $ and $ \io_r \colon \Ss^n \to \Ss^{n+1},~p
	\mapsto \sin r\, p - \cos r\, e_{n+2} $. Then
	\begin{equation*}
		\Psi \colon \SO_{n+2} \times_{\SO_{n+1}} \Diff_+(\Ss^n) \to \sr
		F(\Ss^{n+1};(0,+\infty)),\quad [Q,g] \mapsto Q \circ \iota_r \circ g
	\end{equation*}
	embeds its domain as a deformation retract. 
	
	Given $ f \in \sr F(\Ss^{n+1};(0,+\infty)) $, let $ c_f \in \Ss^{n+1} $
	denote the center of $ f(\Ss^n) $. Choose $ Q_f \in \SO_{n+2} $ with $
	Q_f(-e_{n+2}) = c_f $ and let $ g_f \colon \Ss^n \to \Ss^n $ be the Gauss
	map of $ \pi \circ Q_f^{-1} \circ f $, where $ \pi $ denotes central
	projection of the southern hemisphere onto $ \E^{n+1} \equiv \R^{n+1} \times
	\se{-1} $. Then a homotopy inverse of $ \Psi $ is:
	\begin{equation*}
		\Phi \colon \sr F(\Ss^{n+1};(0,+\infty)) \to \SO_{n+2}
		\times_{\SO_{n+1}} \Diff_+(\Ss^n),\quad f \mapsto [Q_f,g_f].
	\end{equation*}
\end{prp}
\begin{proof}
	The proof consists of five steps.

	\ssk \noindent \tit{Step 1:} $ \Psi $ is well-defined and continuous. \ssk

By \lref{L:fr}\,(b) applied to the canonical inclusion $ \iota \colon \Ss^n
\inc \Ss^{n+1} $, $ \iota_r $ has principal radii everywhere equal to $
r $. Thus $ \iota_r $ lies in $ \sr F(\Ss^{n+1};
(0,+\infty)) $, and so does $ Q \circ \iota_r \circ g $ for any $ Q \in
\SO_{n+2} $ and $ g \in \Diff_+(\Ss^n) $.  Moreover, $ \iota_r $ commutes with
elements of $ \SO_{n+1} $, hence $ \Psi $ is well-defined; it is clear that it
is continuous. Note also that $ \Psi $ is, up to homotopy, independent of the
choice of $ r $.

\ssk \noindent \tit{Step 2:} $ \Phi $ is well-defined and continuous. \ssk 

Although $ c_f $ is uniquely determined, $ Q_f $, and hence
also $ \bar f = \pi \circ Q_{f}^{-1} \circ f $ and $ g_f = \nu_{\bar f} $, are
not. Nevertheless, any other choice $ Q_f' \in \SO_{n+2} $ is related to $ Q_f $ by $
Q'_f = Q_f P $ for some $ P \in \SO_{n+1} $. As $ \pi $ commutes with
elements of $ \SO_{n+1} $, the corresponding $ \bar f' $ is given by 
\begin{equation*}
	{\bar{f}}' = \pi \circ P^{-1} \circ Q_f^{-1} \circ f  = P^{-1} \bar f.
\end{equation*}
Therefore, the corresponding Gauss maps are related by 
\begin{equation*}
	g_f' = \nu_{\bar f'} = P^{-1} \circ \nu_{\bar f} = P^{-1} g_f,
\end{equation*} 
so that $ [Q_f,g_f] = [Q'_f,g'_f] $, that is, $ \Phi $
is well-defined. It is also continuous because $ c_f $ depends continuously on
$ f $ and the bundle projection $ \SO_{n+2} \to \SO_{n+2}/\SO_{n+1} \home
\Ss^{n+1} $ admits local cross-sections.

\ssk \noindent \tit{Step 3:} Suppose that the image of $ h \colon \Ss^n \to
\Ss^{n+1} $ is contained in the southern hemisphere. Then $ h \in \sr
F(\Ss^{n+1};(0,+\infty)) $ if and only if $ \pi \circ h \in \sr
F(\E^{n+1};(-\infty,0)) $. In particular,
\begin{equation}\label{E:assertions}
	\bar f = \pi \circ Q_f^{-1} \circ f \in \sr F(\E^{n+1};(-\infty,0))
	\text{\ and \ } g_f \in \Diff_+(\Ss^n) \  \text{if} \  f \in
	\sr F(\Ss^{n+1};(0,+\infty)). 
\end{equation}

It is clear that $ h $ is an immersion if and only if $ \pi \circ h $ is. By
\cite[prop.~2.3]{CarWar}, the sectional curvatures of $ h $ are everywhere
greater than 1 if and only if those of $ \pi \circ h $ are everywhere greater
than 0.  Equivalently through Gauss' equation, the principal curvatures of $ h
$ have the same sign everywhere if and only if the same holds for $ \pi \circ h
$.  Because $ \Ss^n $ is connected, it suffices to compute the sign at a single
point. This will be done by comparison with an appropriate hypersphere. Let $
\rho \in (0,\frac{\pi}{2}) $  be the radius of the smallest closed metric ball
centered at $ -e_{n+2} $
containing the image of $ h $. Then $ h $ and $ \io_\rho $ are tangent at any
point $ q \in \Ss^{n+1} $ where their images intersect.  Further, these images
lie on the same side of the great hypersphere tangent to
them at $ q $, namely, the one which contains $ -e_{n+2} $. Thus, if the
principal curvatures of $ h $ are positive as those of $ \io_\rho $, then the
Gauss maps of $ h $ and $ \iota_\rho $ coincide at (the corresponding preimages
of) $ q $. Hence so do those of $ \pi \circ h $ and $ \pi \circ \iota_\rho $ at
$ \pi(q) $.  Moreover, $ \pi \circ h (\Ss^n) $ is entirely contained in the
closed ball bounded by $ \pi \circ \iota_\rho (\Ss^n) $.  Consequently their
principal curvatures have the same sign at $ \pi(q) $.  But $ \pi \circ
\iota_\rho \colon p \mapsto \tan \rho\,p \ (p \in \Ss^n)$, and it is readily
seen that the principal curvatures of the latter are negative, according to our
convention \dref{D:Gauss}.  Finally, the Gauss map of an immersion into $
\E^{n+1} $ having negative principal curvatures is an orientation-preserving
diffeomorphism (see \cite[lem.~4.2]{Zuehlke2}).  

\ssk \noindent \tit{Step 4:} $ \Phi \circ \Psi = \id $. \ssk 

Let $ Q \in \SO_{n+2} $ and $ g \in \Diff_+(\Ss^n) $ be arbitrary. The center
$ c_f $ of the image of $ f = \Psi([Q,g]) = Q \circ \iota_r \circ g $ is $
Q(-e_{n+2}) $. Choose $ Q_f = Q $ itself.  Then 
\begin{equation*}
	\bar f = \pi \circ Q_f^{-1} \circ f = \pi \circ \iota_r \circ g = \tan r\,j \circ
	g, 
\end{equation*}
where $ j \colon \Ss^n \inc \E^{n+1} $ is the canonical inclusion. As $ g
$ preserves orientation, 
\begin{equation*}
	g_f = \nu_{\bar f} = \nu_{j \circ g} = \nu_{j} \circ g = \id_{\Ss^n}
	\circ \,g = g.
\end{equation*}
Therefore  $ \Phi \circ \Psi ([Q,g]) = [Q,g] $.

\ssk \noindent \tit{Step 5:} $ \Psi \circ \Phi \iso \id $. \ssk

Let $ f \in \sr F(\Ss^{n+1};(0,+\infty)) $ be arbitrary and let $ j_r \colon
\Ss^n \to \E^{n+1} $, $ p \mapsto \tan r\, p $. Then
\begin{equation*}
	\Psi \circ \Phi(f) = Q_f \circ \iota_r \circ g_f = (Q_f \circ \pi^{-1} )
	\circ \pi \circ \iota_r \circ g_f = (Q_f \circ \pi^{-1}) \circ j_r \circ
	\nu_{\bar f} .
\end{equation*}
Furthermore, we can write
\begin{equation*}
	f = (Q_f \circ \pi^{-1}) \circ \pi \circ Q_f^{-1} \circ f = (Q_f \circ
	\pi^{-1}) \circ \bar f.
\end{equation*}
Thus, by step 3, it suffices to construct a homotopy 
\begin{equation*}
	H \colon [0,1] \times \sr F(\E^{n+1};(-\infty,0)) \to \sr
	F(\E^{n+1};(-\infty,0))
\end{equation*}
such that $ H_0(\phi) = \phi $ and $ H_1(\phi) = j_r \circ \nu_{\phi} $ for any
$ \phi \colon \sr F(\E^{n+1};(-\infty,0)) $. It is shown in 
\cite[prop.~4.3]{Zuehlke2} that the most natural homotopy $ H\colon (s,\phi)
\mapsto (1-s) \phi + s (j_r \circ \nu_\phi) $ works. 

It is clear that $ \Psi $ is an embedding onto a closed subspace. Moreover, for
$ r = \frac{\pi}{4} $, any element of the form $ j_r \circ g $ (where $ g \in
\Diff_+(\Ss^n) $) is stationary under $ H $. It follows that steps 4 and 5 yield 
a deformation retraction of $ \sr F(\Ss^{n+1};(0,+\infty)) $ onto the image of $
\Psi $. 
\end{proof}

\begin{proof}[Proof of \tref{T:homotopical}]
	It can be assumed that $ J $ is nondegenerate, for otherwise the assertion
	is almost trivial.  By \lref{L:fr}\,(b), for $ \rho \in (0,\pi) $ such that
	$ \rho < J $, $ f \mapsto f_\rho $ defines a homeomorphism 
	\begin{equation*}
		h \colon \sr F(\Ss^{n+1};\cot J) \to \sr F(\Ss^{n+1}; \cot (J - \rho)).
	\end{equation*} 
	Furthermore, $ h $ is compatible with $ \Psi $ in the sense that if we
	denote by $ \Psi_r $ the map given by \eqref{E:Psir}, then $ h \circ \Psi_r
	= \Psi_{r-\rho} $. Indeed, for any $ Q \in
	\SO_{n+2} $ and $ g \in \Diff_+(\Ss^n) $,
	\begin{equation*}
		h \circ \Psi_r([Q,g]) = h(Q \circ \iota_r \circ g) = Q \circ
		(\iota_r)_{\rho} \circ g = Q \circ \iota_{r-\rho} \circ g =
		\Psi_{r-\rho}([Q,g]). 	
	\end{equation*}
	Thus $ \Psi_r $ is a w.h.e.~with $ \sr F(\Ss^{n+1};\cot
	J) $ if and only if $ \Psi_{r-\rho} $ is a w.h.e.~with $
	\sr F(\Ss^{n+1}; \cot (J - \rho)) $. Choosing $ \rho $ appropriately, one
	obtains a reduction to the case where $ J \subs \big( 0,\frac{\pi}{2} \big)
	$.  Then in the following commutative triangle:
	\begin{equation*}
		\begin{tikzcd}
			\SO_{n+2}\times_{\SO_{n+1}} \Diff_+(\Ss^n) \arrow[swap]{d}{\Psi_r} 
			\arrow{rd}{\Psi_r} & \\
		    \sr F(\Ss^{n+1};\cot J) \arrow[hookrightarrow]{r}{} &
			\sr F(\Ss^{n+1};(0,+\infty))  
		\end{tikzcd}
	\end{equation*}
	the diagonal arrow is a homotopy equivalence by \pref{P:positive}, while
	\pref{P:whe} guarantees that the horizontal arrow is a weak homotopy
	equivalence. Therefore, so is the vertical arrow.
\end{proof}

\section{Hypersurfaces whose duals have image contained in a
hemisphere}\label{S:hemispherical}

The purpose of this section is to provide a proof of \tref{T:homotopical2}.

\begin{dfn}\label{D:zeta}
	For each $ s \in [0,1] $, let
	\begin{equation*}
		\zeta_s \colon \Ss^{n+1} \ssm \se{\pm e_{n+2}} \to \Ss^{n+1} \ssm
		\se{\pm e_{n+2}},\quad q \mapsto 
	\frac{q-s\gen{q,e_{n+2}}e_{n+2}}{\abs{q-s\gen{q,e_{n+2}}e_{n+2}}}.
	\end{equation*}
	Notice that $ \zeta_0 $ is the identity map. It will be convenient to
	regard $ \zeta_1 $ as a map into $ \Ss^n $.
\end{dfn}

\begin{lem}\label{L:zeta}
	Let $ q \in \Ss^{n+1} $ and $ u,\,v \in T\Ss^{n+1}_q $ be mutually
	orthogonal unit vectors. Suppose that $ \an(v,-e_{n+2}) \leq r <
	\frac{\pi}{2} $. Then $ \abs{d(\zeta_s)_q(u)} \geq \cos r > 0 $ for any $ s
	\in [0,1] $.
\end{lem}
\begin{proof}
	Notice first of all that $ q $ must be distinct from $ \pm e_{n+2} $ (i.e.,
	$ q $ belongs to the domain of $ \zeta_s $), since it is
	orthogonal to $ v$ by hypothesis. A straightforward computation yields that:
	\begin{equation}\label{E:zetas}
		d(\zeta_s)_q(u) = \frac{u-s\gen{u,e_{n+2}}e_{n+2}}
		{\abs{q-s\gen{q,e_{n+2}}e_{n+2}}}
		+ \frac{(2s-s^2)\gen{q,e_{n+2}}\gen{u,e_{n+2}}} 
		{\abs{q-s\gen{q,e_{n+2}}e_{n+2}}^3}
		\big( q-s\gen{q,e_{n+2}}e_{n+2}\big).
	\end{equation}
	After some labor, one obtains that:
	\begin{equation*}
		\abs{d(\zeta_s)_q(u)}^2 = \frac{ 1-(2s-s^2)\big( \gen{q,e_{n+2}}^2 +
		\gen{u,e_{n+2}}^{2} \big)} 
		{ \big[ 1 - (2s - s^2)\gen{q,e_{n+2} }^2  \big]^2 }
	\end{equation*}
	Now for $ s \in [0,1] $, $ (2s - s^2) $ also lies in $ [0,1] $, hence
	\begin{equation*}
		\abs{d(\zeta_s)_q(u)}^2 \geq 1-\gen{q,e_{n+2}}^2 -
		\gen{u,e_{n+2}}^2 \geq \gen{v,e_{n+2}}^2 \geq \cos^2 r,
	\end{equation*}
	where the third inequality holds by hypothesis, while the second one follows
	from the fact that $ e_{n+2} $ has norm $ 1 $ and $ q,\,u,\,v $ are
	orthogonal unit vectors.
\end{proof}

\begin{lem}\label{L:zeta2}
	Let $ N^n $ be a manifold and $ h \colon N^n \to \Ss^{n+1} $ be an immersion
	which induces a complete Riemannian metric and whose Gauss map
	\tup(dual\tup) has image contained in the metric ball of radius $ r <
	\frac{\pi}{2} $ about $ -e_{n+2} $.  Define $ h_s = \zeta_s \circ h $ \tup($
	s \in [0,1] $\tup).  Then:
	\begin{enumerate}
		\item [(a)] $ h_1 $ is an orientation-preserving diffeomorphism between
			$ N^n $ and $ \Ss^n $.\footnote{With the conventions adopted in
			(\ref{D:Gauss}).}
		\item [(b)] Each $ h_s $, in particular $ h = h_0 $, is an embedding.
		\item [(c)] The image of the Gauss map of $ h_s $ is contained in the
			open hemisphere determined by $ -e_{n+2} $ for each $ s $.
	\end{enumerate}
\end{lem}

\begin{urmk}
	The assertion that $ h_1 $ is a diffeomorphism under these hypotheses
	constitutes \cite[thm.~1.2]{WanXia}. The proof presented here is the same as
	that given by Wang/Xia.  
\end{urmk}

\begin{proof}[Proof of \lref{L:zeta2}]
	Applying \lref{L:zeta} with $ q = h(p) $ and $ v = \nu_h(p) $,
	one concludes immediately from the chain rule that 
	\begin{equation*}\label{E:dh}
		\qquad \abs{d(h_s)_p(u)} \geq \cos r \abs{dh_p(u)}\qquad \text{for all $ p \in
		N,~u \in TN_p $, $ s \in [0,1] $}. 
		\qquad 
	\end{equation*}
	Therefore $ h_s $ is an immersion. In
	particular, $ h_1 \colon N^n \to \Ss^n $ is a local diffeomorphism. Since
	$ N $ is complete with respect to the Riemannian metric $ g $ induced by $ h
	$, it is also complete with respect to the homothetic metric $ \cos^2 r\?g
	$. But when $ N $ is furnished with the latter, $ h_1 $ does not decrease
	distances, hence a standard argument (see \cite[prop.~I.3.28]{BriHae}) shows
	that $ h_1 $ is a covering map. 
	As $ \Ss^n $ is simply-connected, this in turn implies that $ h_1 $ is a
	diffeomorphism. 
	
	Suppose that $ h_s(p) = h_s(q) $ for some $ s \in [0,1) $ and $
	p,\,q \in N $. Taking the inner product of both sides with $ e_{n+2} $,
	a straightforward computation shows that $ \gen{h(p),e_{n+2}} =
	\gen{h(q),e_{n+2}}$, and this, together with $ h_s(p) = h_s(q) $, in turn
	implies that $ h(p) = h(q) $. But then clearly $ h_1(p) = h_1(q) $, so that
	$ p = q $ by the preceding paragraph.  Thus each $ h_s $ is an embedding by
	compactness of $ N = \Ss^n $.
	
	Finally, let $ s \in [0,1] $, $ p \in N $ be arbitrary and $ (u_1,\dots,u_n)
	$ be a positively oriented frame in $ TN_p $. Let $ a\sim b $ indicate
	that $ a $ is a positive multiple of $ b $. Then:
	\begin{alignat*}{9}
		& \det \big( d(h_s)_p(u_1)\,,\,\dots\,,\,d(h_s)_p(u_n)\,,\, -e_{n+2} 
		\,,\, h_s(p)\big) & & \\
		\sim & \det \big(
		d(h_s)_p(u_1)\,,\,\dots\,,\,d(h_s)_p(u_n)\,,\, -e_{n+2} \,,\, h(p)\big) 
		& &\quad \text{(by the definition of $ h_s $)}\\
		\sim & \det \big( dh_p(u_1)\,,\,\dots\,,\,dh_p(u_n)\,,\,
		-e_{n+2} \,,\,h(p)\big) & &\quad \text{(by \eqref{E:zetas}, $ q =
		h(p),\,u=dh_p(u_k) $)}\\ 
		\geq & \cos r \det \big(
		dh_p(u_1)\,,\,\dots\,,\,dh_p(u_n) \,,\,\nu_h(p) \,,\,h(p)\big)  > 0 
		& &\quad \text{(since $ \an (\nu_h(p),-e_{n+2}) \leq r < \tfrac{\pi}{2}$)}
	\end{alignat*}
	It follows that $ \gen{\nu_{h_s}(p),-e_{n+2}} > 0 $, establishing (c).
	Moreover, as $ h_1 $ has image contained in $ \Ss^n \subs \R^{n+1} $, by
	setting $ s = 1 $ above, one deduces that $ h_1 $ is orientation-preserving.
\end{proof}

\begin{dfn}\label{D:H}
	Let $ \sr H $ denote the set of all immersions (or, equivalently by
	\lref{L:zeta2}\,(b), embeddings) of $ \Ss^n $ into $ \Ss^{n+1} $ whose Gauss
	maps have image contained in some open hemisphere depending upon the
	immersion, furnished with the $ C^{\infty} $-topology. 
	
	Note that by \cite[thm.~1.2]{WanXia} (or \lref{L:zeta2}\,(a)), $ \sr H $
	coincides with the space of all complete hypersurfaces of $ \Ss^{n+1} $
	whose Gauss maps have images contained in some strictly convex ball.
\end{dfn}

\begin{proof}[Proof of \tref{T:homotopical2}]
	If $ P \in \SO_{n+1} $, then $ P \circ \iota =
	\iota \circ P $.  This implies that $ \bar\Psi $ is well-defined.
	Similarly, $ \bar\Phi $ is well-defined even though $ Q_f $ and $ g_f $ are
	not, because $ \tau = \zeta_1 $ commutes with elements of $ \SO_{n+1} $, and
	any two choices of $ Q_f $ differ by such an element. Furthermore, $
	\bar\Phi $ is continuous because $ c_f $ depends continuously on $ f $ and
	the bundle projection $ \SO_{n+2} \to \SO_{n+2}/\SO_{n+1} \home \Ss^{n+1} $
	admits local cross-sections. Further, $ g_f $ is indeed an
	orientation-preserving diffeomorphism of $ \Ss^n $ by \lref{L:zeta2}\?(a)
	applied to $ h = Q_f^{-1} \circ f $. 

	Let $ g \in \Diff_+(\Ss^n),~Q\in \SO_{n+2} $ be arbitrary. The Gauss map of
	$ \iota \circ g \colon \Ss^n \to \Ss^{n+1} $ is constant and equal to $
	-e_{n+2} $, as one verifies directly from the definition. Hence if 
	\begin{equation*}
		f = Q \circ \iota \circ g = \bar\Psi([Q,g]),
	\end{equation*} 
	then we may choose $ Q_f = Q $, so that $ g_f  = \tau \circ \iota \circ g =
	g $.  Therefore $ \bar\Phi \circ \bar\Psi $ is the identity map.

	We claim that
	\begin{equation*}
		(s,f) \mapsto f_s = Q_f \circ \zeta_s \circ Q_f^{-1} \circ
		f\qquad (s \in [0,1],~f \in \sr H)
	\end{equation*}
	defines a homotopy connecting $ \id_{\sr H} $ and $ \bar\Psi \circ \bar\Phi $.
	It is clear that $ f_0 = f $ and $ f_1 = \bar\Psi \circ \bar\Phi(f) $ for
	any $ f \in \sr H $. Moreover, taking $ h = Q_f^{-1} \circ f $, it follows
	from \lref{L:zeta2} that each $ f_s $ is an immersion (actually, an
	embedding) whose Gauss map has image contained in the open hemisphere
	determined by $ c_f $. Thus $ f_s \in \sr H $ for each $ s $, and the proof
	is complete.
\end{proof}

\section{Hypersurfaces of spherical space forms}\label{S:spherical}

The purpose of this section is to provide a proof of \tref{T:spherical} and to
justify the interpretation  of $ \sr F(M;I) $ claimed in \dref{D:F}. Some
related open problems are listed at the end.

\begin{proof}[Proof of \tref{T:spherical}]
	Let $ \pr_N \colon \te{N} \to N $ be the universal cover of $ N $ and $
	\te{f} $ be any lift of $ f $, so that 
	\begin{equation}\label{E:diagram}
	\begin{tikzcd}
		\te{N} \rar[dashed]{\te{f}} \dar[swap]{\pr_N} & \Ss^{n+1}
		\dar[]{\pr_M} \\ N^n \rar[swap]{f} & M^{n+1}
	\end{tikzcd}
	\end{equation}
	commutes.  Since $ \pr_M $ is a local isometry by hypothesis, the principal
	curvatures of $ \te{f} $ and $ f $ at corresponding points coincide;
	moreover, because $ N $ is complete, so is $ \te{N} $ (with respect to the
	metric induced by $ \te{f} $). Hence $ \te{f} $ satisfies the hypotheses of
	\tref{T:rigidity}, so that $ \te{N} $ must be diffeomorphic to $ \Ss^n $.

	Suppose that $ J(f) = J(\te{f}) $ does not contain $ 0 \pmod \pi $. Choose $
	q \in f(N) $ such that $ |{f^{-1}(q)}| = m $ is as large as possible.  Then,
	by commutativity of \eqref{E:diagram},
	\begin{equation*}
		m \abs{\pi_1(N)} = \abs{\pr_N^{-1}f^{-1}(q)} =
		\big\vert {\te{f}^{-1}\pr_M^{-1}(q)}\big\vert \leq \abs{\pr_M^{-1}(q)} =
		\abs{\pi_1(M)},
	\end{equation*}
	where the inequality comes from the fact that $ \te{f} $ is 
	injective, as guaranteed by \tref{T:rigidity}\,(b). 
	
	\hypertarget{paragraph}{We} claim that the dual of $ \te{f} $ is a lift of the dual of $ f $.
	To see this, express $ N $ as the quotient of its universal cover $ \te{N} $
	by a free proper action of a group $ H < \Iso_+(\te{N})  $. Similarly,
	express $ M $ as a quotient of $ \Ss^{n+1} $ by some $ G < \SO_{n+2} $. By
	connectedness of $ N $ and commutativity of \eqref{E:diagram}, for each $ h
	\in H $ there exists $ g \in G $ such that $ \te{f} \circ h = g \circ \te{f}
	$. Therefore
	\begin{equation*}
		\nu_{\te{f}} \circ h = \nu_{\te{f} \circ h} = \nu_{g \circ \te{f}} = g
		\circ \nu_{\te{f}}\,,
	\end{equation*}
	where in the last equality the fact that $ g \in \SO_{n+2} $ was used.
	As $ \pr_M $ is a local isometry, $ \exp_M \circ\, d(\pr_M) = \pr_M
	\circ \exp_{\Ss^{n+1}} $.  Consequently $ \te{f}^\star := (\te{f})^{\star} $
	factors through $ N $, and its quotient is exactly $ f^{\star}$. In other
	words, 
	\begin{equation*}
	\begin{tikzcd}
		\te{N} \rar{{\te{f}^{\star}}} \dar[swap]{\pr_N} & \Ss^{n+1}
		\dar[]{\pr_M} \\ N^n \rar[swap]{f^{\star}} & M^{n+1}
	\end{tikzcd}
	\end{equation*}
	commutes, as claimed. 
	
	Suppose now that $ J(f) $ does not contain $ \frac{\pi}{2} \pmod \pi $. 
	By \cref{C:dual}\,(b) (applied in the case where the ambient is $ \Ss^{n+1} $), 
	\begin{equation*}
		J(f^\star) = J(\te{f}^{\star}) \equiv  \tfrac{\pi}{2} \pm J(\te{f}) 
		\equiv \tfrac{\pi}{2} \pm J(f) \pmod \pi.
	\end{equation*}
	Moreover, $ N $ is complete with respect to the metric induced by $ f^\star $
	because $ \te{N} $ is complete with respect to the metric induced by $
	\te{f}^\star $ (by \rref{R:lastparagraph} applied to $ \te{f} \colon \te{N}
	\to \Ss^{n+1} $). Thus $ f^{\star} $ satisfies the hypotheses of (a), and
	this implies (b).

	Suppose now that $ f $ is an embedding. Then $ \pr_M^{-1}(f(N)) $ is an
	embedded submanifold of $ \Ss^{n+1} $. Let $ C $ be any of its connected
	components and $ G_C = \set{g \in G}{g(C) = C} $. The set inclusion $ C \inc
	\Ss^{n+1} $ has the same principal curvatures as $ f $, hence $ C $ must be
	diffeomorphic to $ \Ss^n $ by \tref{T:rigidity}. As $ f(N) $ is
	the quotient of $ C $ by $ G_C $, 
	\begin{equation*}
		k\abs{\pi_1(N)} = k\abs{\pi_{1}(f(N))} = k\abs{G_C} = \abs{G} =
		\abs{\pi_1(M)}. \qedhere
	\end{equation*}
\end{proof}

\begin{rmk}[irreducibility of hypersurfaces and $ \sr F(M;I) $]\label{R:meta}
	Let $ N^n $ and $ \bar N^n $ be arbitrary smooth manifolds and $ M^{n+1} $
	an arbitrary Riemannian manifold. Let us call a hypersurface $ \bar f
	\colon \bar N \to M $ a \tdef{factor} of $ f \colon N \to M $ if there
	exists a covering map $ \pr \colon N \to \bar N $ such that $ f = \bar f
	\circ \pr $.  Our viewpoint here is that $ f $ and $ \bar f $ are
	essentially the same object. Define $ f $ to be \tdef{irreducible} if it has
	no factors arising from nontrivial covering maps. 
	
	Now suppose that $ N $ is closed, and furnish it
	with an arbitrary metric space structure compatible with its topology. If $
	f \colon N^n \to M^{n+1} $ is a hypersurface, then there exists $ \eps > 0 $
	such that the restriction of $ f $ to any ball of radius $ \eps $ is
	injective. Set 
	\begin{equation*}
		G_f = \set{ \ga \in \Diff(N) }{ f \circ \ga = f }.
	\end{equation*} 
	The displacement function $ p \mapsto d(p,\ga(p)) $ of an element $ \ga $ of
	$ G_f $ cannot take on any positive value less than $ \eps $ by the choice
	of $ \eps $, hence $ G_f $ acts freely on $ N $. Moreover, if $ {N} $ can be
	covered by $ m $ balls of radius $ \eps $, then the order of $ G_f $ is
	bounded above by $ m $.  We conclude that any hypersurface $ f \colon {N}
	\to M $ has a unique irreducible factor, namely, that induced by $ f $ on $
	{N}/G_f $.

	Returning to our usual context, let $ M^{n+1} $ be a spherical space form
	and $ J \subs \R \pmod \pi $ be an interval of length $ < \frac{\pi}{2} $.
	Then the preceding observation together with \tref{T:spherical} yield a
	bijective correspondence between $ \sr F(M;\cot J) $ and the set of all
	complete hypersurfaces of $ M $ with principal curvatures in $ \cot J $,
	modulo the equivalence relation that identifies two hypersurfaces if they
	have a common factor.  The correspondence simply assigns to each element $ f
	\colon \Ss^n \to M $ of $ \sr F(M;\cot J) $ (the equivalence class of) its
	unique irreducible factor.
\end{rmk}

\begin{rmk}[stronger version of \tref{T:homotopicalspherical}]\label{R:stronger}
	Let $ M^{n+1} $ be a spherical space form and $ \Ga $ its fundamental group.
	Then $ \Ga $ acts on $ \Ss^{n+1} $ via deck transformations, which are
	isometries.  Moreover, when regarded in this way as a subgroup of $
	\SO_{n+2} $, $ \Ga $ also acts naturally on $ \SO_{n+2} \times_{\SO_{n+1}}
	\Diff_+(\Ss^n) $ through $ \ga [Q,g] = [\ga\? Q,g] $. With respect to these
	actions, the w.h.e.~ $ \Psi $ in \tref{T:homotopical} becomes $ \Ga
	$-equivariant. Therefore, under the same hypotheses on $ J $ as in that
	theorem,  $ \Psi $ induces a map 
	\begin{equation*}
		\bar\Psi \colon	(\Ga \backslash \SO_{n+2})  \times_{\SO_{n+1}}
		\Diff_+(\Ss^n) \to \sr F(M;\cot J),\quad [\Ga Q,g] \mapsto \pr \circ\, Q
		\circ \iota_r \circ g
	\end{equation*}
	making the following square commutative, where the vertical arrows are
	covering maps:
	\begin{equation*}
		\begin{tikzcd}
			\phantom{\big( \Ga \backslash \big)}\SO_{n+2} \times_{\SO_{n+1}}
			\Diff_+(\Ss^n) \rar{\Psi} \dar & \sr F(\Ss^{n+1};\cot J)
			\dar{\pr_\ast} \\ \big( \Ga\backslash \SO_{n+2} \big)
			\times_{\SO_{n+1}} \Diff_+(\Ss^n) \arrow[swap]{r}{\bar\Psi} & \sr
			F(M;\cot J)
		\end{tikzcd}
	\end{equation*}
	Looking at the short exact sequences relating the fundamental groups and
	applying the five-lemma, one deduces that $ \bar \Psi $ must also be a
	w.h.e.. 
\end{rmk}

\begin{qtn}
	Let $ J \subs (0,\pi) $ be an interval of length greater than $
	\frac{\pi}{2} $. Suppose that the closed manifold $ N^n $ can be immersed in
	$ \Ss^{n+1} $. Can it be immersed with principal curvatures in $ \cot (J) $?
	More generally, is the inclusion of the space of immersions $
	N^n \to \Ss^{n+1} $ with principal curvatures in $ \cot (J) $ into the space
	of all immersions a homotopy equivalence?
\end{qtn}

\begin{qtn}\label{Q:plane}
	Let $ J \subs \R \pmod \pi $ be an interval of length $ < \frac{\pi}{2} $
	containing  $ 0 \pmod \pi $. What is the homotopy type of $ \sr
	F(\Ss^{n+1};\cot J) $? (Compare \eref{E:plane} and \rref{R:know}.) 
	
	In particular, is the existence of a $ p $-dimensional distribution over $
	\Ss^n $ sufficient to guarantee the existence of a hypersurface in $ \sr
	F(\Ss^{n+1};\cot J) $ with exactly $ p $ positive principal curvatures at
	every point?  
\end{qtn}

\begin{qtn}
	Let $ n > 2 $ be odd. Then in the situation of \tref{T:spherical}, the
	theorem guarantees that $ N^n $ has $ \Ss^n $ for its universal cover.  Must
	$ N $ be diffeomorphic to a spherical space form? Compare \cite{Petrie}.  
\end{qtn}

\subsection*{Acknowledgements}
The author thanks the Mathematics Department of the University of Bras\'ilia
(Brazil) for hosting him as a post-doctoral fellow, and E.~Longa and J.~Ripoll
for clarifications concerning \cite{LonRip1}.  Financial support by
\ltsc{capes} is gratefully acknowledged.

\vfill\eject

\bibliographystyle{amsplain}
\bibliography{references}

\vspace{12pt} \noindent{\small 
\tsc{Departamento de Matem\'atica, Universidade de Bras\'ilia (\ltsc{unb}) \\
Campus Darcy Ribeiro, 70910-900 -- Bras\'ilia, DF, Brazil}} \\
\vspace{-10pt}
\newcommand*{\emailimg}[1]{%
	\raisebox{12pt}{%
    \includegraphics[
		height=12pt,
      keepaspectratio,
    ]{#1}%
  }%
}

\noindent \raisebox{15pt}{\tit{E-mail address:}} \emailimg{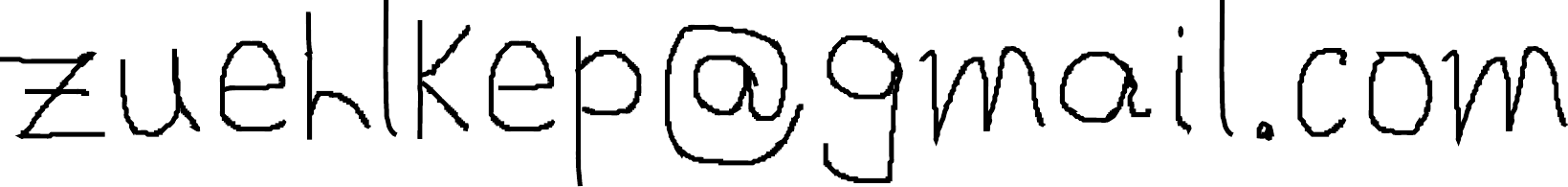}

\end{document}